\documentclass{article}

\usepackage{PRIMEarxiv}

\usepackage[utf8]{inputenc} 
\usepackage[T1]{fontenc}    
\usepackage{hyperref}       
\usepackage{url}            
\usepackage{booktabs}       
\usepackage{amsfonts}       
\usepackage{nicefrac}       
\usepackage{microtype}      
\usepackage{lipsum}
\usepackage{fancyhdr}       
\usepackage{graphicx}       
\graphicspath{{media/}}     
\usepackage{amsmath}
\usepackage{amssymb}
\usepackage{amsthm}
\usepackage{comment}
\usepackage{xcolor}
\usepackage{bbm}

\theoremstyle{definition}
\newtheorem{definition}{Definition}[section]
\newtheorem{remark}{Remark}
\newtheorem{lemma}{Lemma}
\newtheorem{theorem}{Theorem}
\newtheorem{proposition}{Proposition}

\title{Random Processes with Stationary Increments and Intrinsic Random Functions on the Real Line
}

\author{
  Jongwook Kim \\
  Ball State University \\
  \texttt{jongwook.kim@bsu.edu} \\
}

\begin{document}
\maketitle

\begin{abstract}
Random processes with stationary increments and intrinsic random processes are two concepts commonly used to deal with non-stationary random processes. They are broader classes than stationary random processes and conceptually closely related to each other. This paper illustrates the relationship between these two concepts of stochastic processes and shows that, under certain conditions, they are equivalent on the real line.
\end{abstract}

\keywords{intrinsic random function \and intrinsic covariance function \and random process with stationary increment \and integrated process \and stochastic process \and time series \and spatial statistics}

\section{Introduction}\label{sec:intro}

Although stationary assumption is widely used to analyze random processes, it is often too strict in real world. To relax the assumption of stationarity, Yaglom (1958) \cite{Yaglom1958} introduced a random process with stationary increments in which distribution of increments of random variables at different temporal or spatial points only depend on time lags or distances. He also developed the spectral representations of the processes with stationary n-th increments and their structure functions, which are essential for understanding the covariance structures of the processes. When the process is intrinsic stationary, the structure function can be seen as variogram that can replace the covariance function to analyze the spatial or temporal dependencies of the stochastic processes (Cressie, 1993) \cite{cressie1993}. Moreover, notion of structure function and its application is essential in the prediction model such as universal kriging (Zhang \& Huang, 2014) \cite{ZHANG2014153}. Given their practicality and usefulness, they are broadly used across various fields, especially time series analysis and econometrics. 

To address non-stationarity for similar purposes, the concept of an intrinsic random function was suggested by Matheron (1973) \cite{Materon1973}. Such random processes are characterized by their intrinsic properties under more abstract and profound mathematical settings. Therefore, compared to a random process with stationary increments, an intrinsic random function is more flexible and applicable not only to the real line but also to various other mathematical spaces. For example, Matheron (1973) \cite{Materon1973} illustrated how intrinsic random functions (IRFs) can be defined in higher-dimensional Euclidean spaces. He demonstrated that the mean of the IRF consists of lower monomials and that the transformed process becomes stationary. Similarly, Huang et al. (2016; 2019) \cite{HUANG201633} \cite{HUANG20197} extended the concept of intrinsic random functions to the circle and sphere, highlighting the critical role of low-frequency truncation to achieve stationarity. This flexibility of the IRF(d) framework made it most prevalent in geostatistics and spatial statistics.

Despite their frequent co-application to non-stationary stochastic problems, there is insufficient literature clearly explaining the theoretical link between the Intrinsic Random Function (IRF(d)), rooted in geostatistics, and the Random Process with Stationary Increments (I(d)), common in time series, on the real line. The central finding of this paper is the demonstration that these two process classes are, in fact, equivalent under specified conditions. This equivalence provides a clear, integrated mathematical foundation, resolving the need to characterize a non-stationary process as one or the other. Consequently, it strengthens the definition of these processes on $\mathbb{R}$ and facilitates the powerful cross-pollination of methods, such as applying I(d)'s spectral theory to IRF(d) problems, and vice-versa.

\bigskip

\section{Background and Definitions}

To formally establish this relationship, we first begin by reviewing the fundamental mathematical definitions and properties of the two concepts, starting with the random process with stationary increments.

\bigskip

\subsection{Random Processes with Stationary Increments} \label{subsec:rpsi}

\begin{definition} \label{def:random_process_stationary_increment}
(Yaglom 1958; Wei Chen et al. 2024) \cite{Yaglom1958} \cite{Chen2024} \\ 
$\{X(t), t \in \mathbb{R}\}$ is called a random process with stationary increments of order $d \in \mathbb{Z}$ $(d \ge 0)$, denoted I(d), if $\Delta_{\iota}^{d}X(t)$ is stationary where \\
$$\Delta_{\iota}^{d}X(t) = \sum_{k=0}^{d} (-1)^k \binom{d}{k} X(t - k\iota), \quad t,\iota \in \mathbb{R} \quad d \in \mathbb{Z}_{\ge 0}$$
Additionally, $D^d(h; \iota_1, \iota_2)$ is its structure function where\\ 
$$D^d(h;\iota_1,\iota_2) = E\bigl([\Delta_{\iota_1}^{d}X(t+h)] [\Delta_{\iota_2}^{d}X(t)]\bigl), \quad h,\iota_1,\iota_2 \in \mathbb{R}$$
This is positive semi-definite. 
\end{definition}

Note that when $\iota=\iota_1=\iota_2$ and $h=0$, $D(\iota) = E\bigl( [X(t) - X(t-\iota)]^2 \bigl)$ is a variogram. If $d=1$, then $\Delta_{\iota}X(t) = X(t) - X(t-\iota)$; thus, $D(\iota_1,\iota_2) = E\bigl( [X(t)-X(t - \iota_1)] \cdot [X(t)-X(t - \iota_2)] \bigl)$. Likewise, if $d=2$, $\Delta_{\iota}^2 X(t) = X(t) - 2 X(t-\iota) + X(t-2\iota) = \{X(t) - X(t-\iota)\} - \{ X(t-\iota) - X(t-2\iota) \}$. A stationary process corresponds to the case when $d=0$. When $d=1$, the random process is intrinsically stationary in that $X(t)-X(t-1)$ is stationary. For a discrete time series, that is, $t \in \mathbb{Z}$, the random process with stationary increments of order $d$ is also called an integrated process. According to Brockwell and Davis (1991) \cite{Brockwell1991}, the order $d$ of stationary increments usually stops at 1 or 2 in practice.

\bigskip

\begin{remark}\label{rmk:spectral_decomp_conti} (Yaglom 1987) \cite{yaglom1987correlation},
Let $X(t)$ be a random process with stationary increments with order $d$, and $Y(t)$ be a stationary random process such that $Y=X^{(d)}$, which is a d-th derivative (or generalized derivative) of $X$, where $t,h \in \mathbb{R}$. Since $Y(t)$ is stationary,

\begin{align}
&Y(t) =  \int_{-\infty}^{\infty}e^{i \omega t}dZ_y(\omega), \quad E \bigl(Y(t+h) \cdot Y(t) \bigl) = C_y(h) = \int_{-\infty}^{\infty}e^{i \omega h}dF_y(\omega), \label{eq:corr_y}
\end{align}

where $F_y(\omega)$ is a bounded non-decreasing function and $E(|dZ_y(\omega)|^2) = dF_y(\omega)$ for $Z_y(\omega)$, a random function with uncorrelated increments. Then, we can achieve the following spectral representations:
\begin{align}
X(t) &=  \int_{-\infty}^{\infty} \bigl( e^{it\omega} -1 - it\omega - \cdots - \frac{(i t \omega)^{d-1}}{(d-1)!} \bigl) dZ(\omega) + X_0 + X_1 t + \cdots + X_d t^d \label{eq:spectra_I(d)}\\
\Delta_{\iota}^d X(t) &= \int_{-\infty}^{\infty} e^{i\omega t} (1-e^{-i \iota\omega})^d dZ(\omega) + (d!) X_d \iota^d \label{eq:spectral_differenced} \\
D(h; \iota_1,\iota_2) &= \int_{-\infty}^{\infty} e^{i\omega h} (1-e^{-i \omega \iota_1})^d (1-e^{i \omega \iota_2})^d dF(\omega) + (d!)^2 A_d^2 \iota_1^d \iota_2^d \label{eq:spectral_structur_fn}
\end{align}
Where  $\int_{-\infty}^{\infty} = \lim_{T \rightarrow \infty, \epsilon \rightarrow 0} \{ \int_{-T}^{-\epsilon} + \int_{\epsilon}^{T}\}$, \quad $Z(\omega_2) - Z(\omega_1) = \int_{-\omega_1}^{\omega_2} \frac{dZ_y(\omega)}{(i \omega)^d}$, \quad and \quad $F(\omega_2) - F(\omega_1) = \int_{-\omega_1}^{\omega_2} \frac{dF_y(\omega)}{\omega^{2d}}$.  A random variable $X_0 = X(0)$ and $X_1, X_2, \cdots, X_{d-1}$ are some random variables. A random variable $X_d$ is  introduced in order to deal with the jump discontinuity at $\omega=0$. These random variables $X_0, X_1, \cdots, X_d$ can, in particular, be constants. In the structure function, $\langle|X_d| \rangle = A_d^2 \ge 0$.
\end{remark}

\bigskip

\begin{remark}\label{rmk:spectral_decomp_conti2} (Yaglom 1987) \cite{yaglom1987correlation}
If the correlation function of $C_y(h)$ in \eqref{eq:corr_y} decays rapidly with respect to $h$, that is,
\begin{align}
    &\int_{-\infty}^{\infty} |C_y(h)|^2 dh < \infty \label{eq:corr_y_decay}
\end{align}
then, the random variable $X_d$ and $A_d$ in \eqref{eq:spectra_I(d)}, \eqref{eq:spectral_differenced}, and \eqref{eq:spectral_structur_fn} are zero, and the spectral density function of the process $X(t)$ exists; thus, the spectral distribution function $dF(\omega)$ in \eqref{eq:spectral_structur_fn} can be expressed as $f(\omega)d\omega$ where $f(\omega) = F'(\omega)$. In the later part of this research, we assume the condition in \eqref{eq:corr_y_decay} is satisfied. We also assume that the random variables $X_0, X_1, \cdots, X_{d-1}$ are zero for simplicity.
\end{remark}

\bigskip

\begin{remark}\label{rmk:delta} (Chiles and Delfiner 1999) \cite{Chiles1999}
The differencing operator $\Delta_{\iota}$ in Definition \ref{def:random_process_stationary_increment} reduces the degree of the polynomial by one level. 
That is, if we suppose any polynomial function with any degree $n \in \mathbb{N}$ such as 
$$p_n(t) = \sum_{i=0}^{n} a_{i} t^i$$
Then,
$$\Delta_{\iota} p_n(t) = \sum_{i=1}^{n} a_i p^*_{i-1}(t) \quad \text{where} \quad p^*_{i-1}(t) = -\sum_{j=0}^{i-1} \binom{i}{j} t^{j} (-\iota)^{i-j}$$
\end{remark}

\bigskip

\subsection{Intrinsic Random Functions on the Real line}\label{subsec:irf}

In this section, the definition and properties of an IRF on the real line $\mathbb{R}$ will be examined. For this aim, it is required to identify its allowable measure.

\begin{definition} (Matheron \@1973 ; Chiles and Delfiner \@ 1999) \cite{Materon1973} \cite{Chiles1999} \label{def:allowable_measure_temporal}
Let $\Lambda_d$ be the space of measures with compact supports and $\lambda \in \Lambda_d$. $\lambda$ is an allowable measure of order $d$ on $\mathbb{R}$ if it annihilates polynomials of degree less than $d$. That is,

$$
f(\lambda) = \int f(x) \lambda(dx)
$$

and for a polynomial function $p(\cdot)$ of degree less than $d$,

$$
p(\lambda) = \int p(x) \lambda(dx) = 0 
$$

where $\ell=0,1,2,\cdots, d-1$. We denote that $\Lambda_d$ is a class of such allowable measures, and then it is clear that $\Lambda_{d+1} \subset \Lambda_{d}$.
\end{definition}

\bigskip

A common example is the discrete allowable measure formed by a finite set of weights $\lambda_i$ at points $x_i$ on $\mathbb{R}$:
$$
\lambda = \sum_{i=1}^{n} \lambda_i \delta_{x_i} \quad \text{where $\delta_{x_i}$ is the Dirac measure at $x_i$}
$$
In this case, the integral becomes a summation:
$$
f(\lambda) = \int f(x) \lambda(dx) = \sum_{i=1}^{n} \lambda_i f(x_i)
$$
The annihilation condition for a polynomial $p(x)$ of degree less than $d$ then requires:
\begin{align*}
p(\lambda) &= \int p(x) \lambda(dx) \\
&= \sum_{i=1}^{n} \lambda_i p(x_i) = 0
\end{align*}

\bigskip

\begin{remark} (Chiles and Delfiner 1999) \cite{Chiles1999} \label{rmk:allowable_measure_shift invariant}
An allowable measure $\lambda \in \Lambda_d$ is closed under translations of shift. That is,
$$p_{d-1}(\tau_h \lambda) = \int p_{d-1}(x+h) \lambda(dx) = 0$$
where $p_{d-1}(\cdot)$ is a polynomial function of degree $d-1$ and $\tau_h$ is a shift operator by $h \in \mathbb{R}$.
\end{remark}

\bigskip


\begin{definition} \label{def:irf_temporal}
A random process $X(t)$ on the real line is an intrinsic random function with its order $d$, denoted IRF(d), if the process $X(\lambda)$ is shift invariant for any $t \in \mathbb{R}$ and any $\lambda \in \Lambda_d$. That is, for any allowable measure $\lambda \in \Lambda_d$ and any $\tau_h \in \mathrm{G}$ where $\mathrm{G}$ is a group of shift operators and $h \in \mathbb{R}$, it satisfies that
\begin{align*}
&E\bigl(X(\lambda) \bigl) = E\bigl(X(\tau_h \lambda) \bigl)\\
&Cov\bigl(X(\lambda_1), X(\lambda_2) \bigl) = Cov\bigl(X(\tau_h \lambda_1), X(\tau_h \lambda_2) \bigl) \quad \text{where} \quad \text{any }\lambda_1, \lambda_2 \in \Lambda_d
\end{align*}
\end{definition}

\bigskip

\begin{remark}
    A stationary random process is denoted as an IRF(0). Note that this notation differs slightly from Matheron’s (1793) \cite{Materon1973}, where a stationary process is represented by IRF(-1). In this paper, $\kappa \ge 1$ is used to describe nonstationary random processes, instead of $\kappa \ge 0$.
\end{remark}

\bigskip

\begin{definition} \label{def:icf}
    Let a random process $X(t)$ be an IRF(d) and let $\Lambda_d$ be the space of allowable measures of order $d$. The intrinsic covariance function $K(\cdot)$ of $X(t)$ is the stationary function that defines the covariance between any two allowable linear combinations, $X(\lambda_1)$ and $X(\lambda_2)$, for all $\lambda_1, \lambda_2 \in \Lambda_d$, such that:
    $$
    Cov\bigl(X(\lambda_1), X(\lambda_2)\bigl) = \int_{\mathbb{R}} \int_{\mathbb{R}} K(x-y) \, \lambda_1(dx) \, \lambda_2(dy).
    $$
    This stationary function $K(\cdot)$ is called the Intrinsic Covariance Function ICF of order $d$, and is also widely referred to as a Generalized Covariance Function (Matheron 1793; Chiles and Delfiner 1999) \cite{Materon1973} \cite{Chiles1999}..
\end{definition}

\bigskip

\section{Connection between Random Processes with Stationary Increments and Intrinsic Random Functions on $\mathbb{R}$}

Finally, we explore the relationship between random processes with stationary increments and intrinsic random functions, and show they are equivalent on the real line. By using the definition \ref{def:allowable_measure_temporal} and \ref{def:irf_temporal}, we can easily show that the difference operator used for random processes with stationary increments is in fact an allowable measure for an IRF(d).

\bigskip

\begin{lemma}\label{lemm:allowable_measure_delta}
Define
\[
\lambda_{\Delta_{\iota,t}^d} := \sum_{k=0}^{d} (-1)^k \binom{d}{k} \delta_{t - k\iota}, \quad t,\iota \in \mathbb{R}.
\]
Suppose that \( X(t) \) is an IRF$(d)$. Then, \( \lambda_{\Delta_{\iota,t}^d} \in \Lambda_d \), and \( X(\lambda_{\Delta_{\iota,t}^d}) = \Delta_{\iota}^d X(t) \). That is, the finite difference operator \( \Delta_{\iota}^d \), defined in Definition~\ref{def:random_process_stationary_increment}, corresponds to an allowable measure for an IRF$(d)$ and annihilates all polynomials of degree up to \( d-1 \).
\end{lemma}

\begin{proof}
First, suppose $p_{d-1}(t)$ is a polynomial of degree $d-1$. Since $p_{d-1}(\lambda_{\Delta_{\iota,t}})=\Delta_{\iota} p_{d-1}(t)$, by Remark \ref{rmk:delta}, $\Delta_{\iota} p_{d-1}(t)$ becomes a polynomial of order $d-2$. Then, it is obvious that $\Delta_{\iota}^{d-1}p_{d-1}(t)$ is a polynomial with degree 0, which is a constant.\\
Since 
 $$ p_{d-1}(\lambda_{\Delta_{\iota,t}^d}) = \Delta_{\iota}^{d}p_{d-1}(t) = \Delta_{\iota} \biggl\{\Delta_{\iota}^{d-1} p_{d-1}(t) \biggl\},$$
we can conclude that
$$\Delta_{\iota}^{d} p_{d-1}(x) = 0$$

To sum up, $\Delta_{\iota}^{d}$ annihilates polynomials of degree $d-1$; thus, it is an allowable measure of order $d$.
\end{proof}

\bigskip

\begin{definition}\label{def:Stationarity}
Suppose that $X(\cdot)$ is a random process on the real line $\mathbb{R}$. Then, the process is called (weakly) stationary if 

$$E\biggl( X(t) \biggl) = E\biggl( X(\tau_h t) \biggl)$$
and
$$Cov\biggl( X(t), X(s) \biggl) = Cov\biggl( X(\tau_h t), X(\tau_h s) \biggl)$$
for any $t,s,h \in \mathbb{R}$ and any $\tau_h \in \mathrm{G}$ where $\mathrm{G}$ is a group of shift operators.  
\end{definition}

\bigskip

\begin{lemma}\label{lemm:stationary}
Suppose $Y(t)$ is stationary for $t \in \mathbb{R}$ and continuous in quadratic mean. Then, $Y(\lambda)$ is stationary for any allowable measure with any order. That is,
$$E\bigl(Y(\lambda) \bigl) = E\bigl(Y(\tau_h\lambda) \bigl)$$
$$ Cov\bigl(Y(\lambda_1), Y(\lambda_2) \bigl) = Cov\bigl(Y(\tau_h \lambda_1), Y(\tau_h \lambda_2) \bigl)$$
for any $\lambda, \lambda_1, \lambda_2 \in \Lambda_d$ where $d$ is any positive integer, and $\tau_h$, a shift operator with any $h \in \mathbb{R}$.
\end{lemma}
\begin{proof}
Suppose that the random process $Y(t)$ is stationary for $t \in \mathbb{R}$, as in Definition~\ref{def:Stationarity}, and let $\lambda \in \Lambda_d$ be any allowable measure with compact support, where $d \in \mathbb{Z}_+$.

\begin{align*}
    E\bigl( Y(\lambda) \bigl) &= E\bigl( \int Y(t) \lambda(dt) \bigl)\\
    &= \int E\bigl( Y(t) \bigl) \lambda(dt)
\end{align*}
since the order of integration and expectation can be interchanged by the Fubini's theorem in that $\lambda$ has compact support (finite measure) and $Y(t)$ is continuous in quadratic mean (implying integrability in the $L^2$ sense). Then, by stationarity of $Y(t)$ in terms of t, for any $h\in \mathbb{R}$, we can also state that
\begin{align*}
   \int E\bigl( Y(t) \bigl) \lambda(dt) &= \int E\bigl( Y(t+h) \bigl) \lambda(dt)\\
    &= E\bigl( \int Y(t+h) \lambda(dt) \bigl) = E\bigl( Y(\tau_h \lambda) \bigl)\\
\end{align*}
since $\tau_h \lambda$ is still a finite measure.

Similarly, for any $\lambda_1, \lambda_2 \in \Lambda_d$ with compact support, $d \in \mathbb{Z}_+$, and $t,s \in \mathbb{R}$,
\begin{align*}
    Cov\bigl( Y(\lambda_1), Y(\lambda_2) \bigl) &= Cov\bigl( \int Y(t) \lambda_1(dt), \int Y(s) \lambda_{2}(ds) \bigl)\\
    &= \int \int Cov\bigl( Y(t), Y(s) \bigl) \lambda_1(dt) \lambda_2(ds) \quad \text{(by Fubini's thoerem with the same reasons)}
\end{align*}
By stationarity of $Y(t)$ in terms of any $t \in \mathbb{R}$, for any $h\in \mathbb{R}$,
\begin{align*}
    Cov\bigl( Y(\lambda_1), Y(\lambda_2) \bigl) &= Cov\bigl( \int Y(t) \lambda_1(dt), \int Y(s) \lambda_{2}(ds) \bigl)\\
    &= \int \int Cov\bigl( Y(t+h), Y(s+h) \bigl) \lambda_1(dt) \lambda_2(ds)\\
    &= Cov\bigl(\int Y(t+h) \lambda_1(dt), \int Y(s+h) \lambda_2(ds) \bigl)\\
    &= Cov\bigl( Y(\tau_h \lambda_1), Y(\tau_h \lambda_2) \bigl)
\end{align*}
Therefore, if $Y(t)$ is stationary for any $t \in \mathbb{R}$, then $Y(\lambda)$ is stationary for $\lambda$, an allowable measure with any order.
\end{proof}

\bigskip

With the definitions, remarks and lemmas provided, we can now establish the connection between intrinsic random functions and random processes with stationary increments on the real line.

\bigskip

\begin{theorem}\label{thm:irf_rpsi}
For $d \ge 1$ and $d \in \mathbb{Z}$, a stochastic process $X(t)$ is an intrinsic random function of order $d$ (IRF(d)) on the real line if and only if $X(t)$ is a random process with stationary increments order $d$ provided that the condition in Remark \ref{rmk:spectral_decomp_conti2}.
\end{theorem}

\begin{proof}

($\Rightarrow$) Let a random process $X(\cdot)$ be an intrinsic random function with order $d$ on the real line, and define\\ 
$$\lambda_{\Delta_{\iota,t}^{d}} := \sum_{k=0}^{d} (-1)^k \binom{d}{k} \delta_{t - k\iota}, \quad t,\iota \in \mathbb{R}.$$\\
Then, $X(\lambda_{\Delta_{\iota,t}^{d}}) = \Delta_{\iota}^{d} X(t)$ and by Lemma~\ref{lemm:allowable_measure_delta}, $\lambda_{\Delta_{\iota,t}^{d}}$ is allowable with its degree $d$ on $\mathbb{R}$.\\

By Definition~\ref{def:irf_temporal} of an intrinsic random function, for an allowable measure $\lambda_{\Delta_{\iota,t}^{d}}, \lambda_{\Delta_{\iota,s}^{d}} \in \Lambda_{d}$ and any $t,s,x \in \mathbb{R}$,
\begin{align}
&E\bigl(X(\lambda_{\Delta_{\iota,t}^{d}}) \bigl) = E\bigl(X(\tau_x\lambda_{\Delta_{\iota,t}^{d}}) \bigl) \label{eq:irf_id1}\\
&Cov\bigl(X(\lambda_{\Delta_{\iota,t}^{d}}), X(\lambda_{\Delta_{\iota,s}^{d}}) \bigl) = Cov\bigl(X(\tau_x \lambda_{\Delta_{\iota,t}^{d}}), X(\tau_x \lambda_{\Delta_{\iota,s}^{d}}) \bigl) \label{eq:irf_id2}
\end{align}

Since $X(\lambda_{\Delta_{\iota,t}^{d}}) = \Delta_{\iota}^{d} X(t)$ and $X(\lambda_{\Delta_{\iota,s}^{d}}) = \Delta_{\iota}^{d} X(s)$, Equation~\eqref{eq:irf_id1} and \eqref{eq:irf_id2} can also be expressed as\\
\begin{align*}
&E\bigl( \Delta_{\iota}^{d} X(t) \bigl) = E\bigl( \Delta_{\iota}^{d} X(\tau_x t) \bigl)\\
&Cov\bigl( \Delta_{\iota}^{d} X(t), \Delta_{\iota}^{d} X(s) \bigl) = Cov\bigl( \Delta_{\iota}^{d} X(\tau_x t), \Delta_{\iota}^{d} X(\tau_x s) \bigl)
\end{align*}

Hence, $\Delta_{\iota}^{d} X(t)$ is stationary, which implies that $X(t)$ is a random process with stationary increments of order $d$.

\bigskip

($\Leftarrow$) Suppose $X_d(t)$ is a random process with stationary increments of order $d$. Then, by its spectral representation given in Remark \ref{rmk:spectral_decomp_conti} and the condition in Remark \ref{rmk:spectral_decomp_conti2},

$$X(t) = \int_{-\infty}^{\infty} \bigl( e^{it\omega} -1 - it\omega - \cdots - \frac{(i t \omega)^{d-1}}{(d-1)!} \bigl) dZ(\omega)$$
Then, for any allowable measure $\lambda \in \Lambda_d$ with compact support,

\begin{align*}
X(\lambda) &= \int_\mathbb{R} \int_{-\infty}^{\infty} \bigl( e^{it\omega} -1 - it\omega - \cdots - \frac{(i t \omega)^{d-1}}{(d-1)!} \bigl) dZ(\omega) \lambda(dt)\\
&= \int_{-\infty}^{\infty} \int_\mathbb{R} \bigl( e^{i t \omega} -1 - it\omega - \cdots - \frac{(i t \omega)^{d-1}}{(d-1)!} \bigl) \lambda(dt) dZ(\omega) \quad \text{(by Fubini's theorem).}\\
\end{align*}
Since $\lambda$ has compact support, the inner integral is finite. Furthermore, because $X(t)$ is a random process (assumed $L^2$) and $\lambda \in \Lambda_d$, $X(\lambda)$ is guaranteed to be in $L^2(\Omega)$, where $\mathbf{\Omega}$ denotes the sample space. This justifies the use of the Fubini's theorem to interchange the order of stochastic and Lebesgue-Stieltjes integration.
By the definition of an allowable measure, any $\lambda \in \Lambda_{d}$ annihilates polynomials of its order $d-1$. Hence,
$$
X(\lambda)=\int_{-\infty}^{\infty} \int_\mathbb{R} e^{i t \omega} \lambda(dt) dZ(\omega) = \int_\mathbb{R} \int_{-\infty}^{\infty} e^{i t \omega} dZ(\omega) \lambda(dt) \quad \text{(by Fubini's theorem again).}
$$
Therefore, by defining $Y^*(t) := \int_{-\infty}^{\infty} e^{it\omega} dZ(\omega)$, which is a stationary process for $t \in \mathbb{R}$,
$$X(\lambda) = \int_\mathbb{R} Y^*(t) \lambda(t) = Y^*(\lambda)$$

Since $Y^*(t)$ is stationary with respect to $t$, $Y^*(\lambda)$ is stationary with respect to $\lambda$ by Lemma \ref{lemm:stationary}. Therefore, $X(\lambda)$ is stationary for $\lambda$ as well.
In conclusion, $X(t)$ is an intrinsic random function of order $d$ on the real line.
\end{proof}

\bigskip

We can also analyze the relationship between intrinsic structure function of IRF and structure function of I(d) process.

\bigskip

\begin{remark} \label{rmk:structure_fn_and_ICF}
    Let $X(\cdot)$ be an intrinsic random function of order $d$ on $\mathbb{R}$, denoted as IRF$(d)$. Let $K(\cdot)$ be its intrinsic covariance function of $X(\cdot)$ in Definition~\ref{def:icf}. Then, the structure function $E\biggl( \Delta_{\iota_1}^{d}X(t) \cdot \Delta_{\iota_2}^{d}X(s) \biggl)$ in Definition~\ref{def:random_process_stationary_increment} can be expressed in terms of the intrinsic covariance function as
    \begin{align*}
        E\bigl( \Delta_{\iota_1}^{d}X(t) \cdot \Delta_{\iota_2}^{d}X(s) \bigl) &= \sum_{k_1=0}^{d} (-1)^{k_1} \binom{d}{k_1} \sum_{k_2=0}^{d} (-1)^{k_2} \binom{d}{k_2} K\bigl( (t - k_1 \iota_1) - (s-k_2 \iota_2) \bigl)
    \end{align*}
\end{remark}
For example, for $d=\iota_1=\iota_2=1$, 
    \begin{align*}
        E\bigl( \Delta^{1}X(t) \cdot \Delta^{1}X(s) \bigl) &= 2K(t-s) - K(t-s+1) - K(t-s-1)
    \end{align*}

\bigskip

The equivalence established by Theorem \ref{thm:irf_rpsi} and explicitly linked by Remark \ref{rmk:structure_fn_and_ICF} confirms that the Intrinsic Covariance Function ($K(\cdot)$) and the Structure Function both represent the same, unique dependency information within the non-stationary process.

Since the overall process $X(t)$ has infinite variance, the stationary information is uniquely limited to the covariance structure of the allowable increments. Therefore, $K(\cdot)$ and the Structure Function share the unique most important stationary information—the precise second-order dependency that remains after the non-stationary polynomial drift has been mathematically annihilated. This is the only component of the process that is stable and estimable from a single realization of the data.

As a result of Theorem \ref{thm:irf_rpsi}, we can conclude that for an $\text{I(d)}$ process $X(t)$, the stationary covariance function of $\Delta_\iota^d X(t)$ is structurally determined by the intrinsic covariance function of order $d$ on the real line. In addition, it enables us to see that the differencing operator $\Delta_\iota^d$ is one of the allowable measures that annihilate the polynomial components of an intrinsic random function. Theorem \ref{thm:irf_rpsi} also allows us to use the spectral representation of $\text{I(d)}$ to explore the structure of an intrinsic random function on the real line.

\bigskip

\section{Applications and Discussions} \label{sec:application_discuss}

So far, we have demonstrated the equivalence between random processes with stationary increments and intrinsic random functions on the real line. This fact allows us to apply the theories of each concept to the other, enhancing our understanding of both.

The theoretical unification provides immediate practical benefits for modelers in both time series and geostatistics. Specifically, the congruence validates geostatistical models (as concepts like the intrinsic covariance function used in IRF models are directly supported by the well-established spectral theory of I(d) processes). It enables spectral estimation by allowing practitioners to apply advanced spectral analysis tools, traditionally limited to time series, to estimate intrinsic properties and fit covariance models in non-stationary settings with greater analytical justification. Finally, it facilitates methodological cross-pollination, resolving the need to characterize a non-stationary process as one or the other.

 In this section, we explore applications and examples that illustrate how the combined theories of random processes with stationary increments and an intrinsic random function can be used to deepen our understanding of non-stationary random processes on the real line.

\subsection{Brownian Motion}
The $d=1$ case, often termed Brownian Motion (or the Wiener process), serves as the canonical example demonstrating the equivalence established in Theorem \ref{thm:irf_rpsi}. As a well-known random process with stationary increments of order I(1), differencing the process induces stationarity. We illustrate how its structure aligns precisely with the properties of an IRF(1) process. For simplicity, we assume $X(0)=X_0=0$, allowing its spectral representation and covariance function $D(t,s) = \mathrm{Cov}(X(t), X(s))$ to be expressed as (Yaglom 1987) \cite{yaglom1987correlation}:
\begin{align*}
&X(t) = \frac{1}{\sqrt{2\pi}} \int_{-\infty}^{\infty} \frac{e^{it\omega} - 1}{i\omega} dZ^*(\omega)\\
&D(t,s) = Cov\bigl( X(t), X(s) \bigl) =
\begin{cases}
C \text{ min} \bigl\{ |t|, |s| \bigl\} \quad &\text{for} \quad ts \ge 0\\
0 \quad &\text{for} \quad ts < 0
\end{cases}
\end{align*}

where $dZ^*(\omega) = \sqrt{2\pi} i \omega dZ(\omega)$ and C is a positive constant.

The equivalence from Theorem~\ref{thm:irf_rpsi} implies that Brownian Motion must also satisfy the definition of an IRF(1). This means that applying any allowable measure $\lambda \in \Lambda_1$ (which annihilates constant terms) to $X(t)$ must result in a stationary process $X(\lambda)$. Applying an arbitrary $\lambda \in \Lambda_1$:
\begin{align*}
X(\lambda) &= \frac{1}{\sqrt{2\pi}} \int_{\mathbb{R}} \int_{-\infty}^{\infty} \frac{e^{i t\omega} - 1}{i\omega} dZ^*(\omega) \lambda(dt)\\
&\text{By Fubini's theorem,}\\
&= \frac{1}{\sqrt{2\pi}}\int_{-\infty}^{\infty} \int_{\mathbb{R}} \frac{e^{i t\omega} - 1}{i\omega} \lambda(dt) dZ^*(\omega)\\
&\text{Since $\lambda \in \Lambda_1$ annihilates the polynomial of degree $d-1=0$, $\int 1 \cdot \lambda(dt) = 0$.}\\
&\text{Thus,the term corresponding to the constant '1' vanishes, leaving:}\\
&= \frac{1}{\sqrt{2\pi}}\int_{-\infty}^{\infty} \int_{\mathbb{R}} \frac{e^{i t\omega}}{i\omega} \lambda(dt) dZ^*(\omega)\\
&\text{By defining } Y(t) := \frac{1}{\sqrt{2\pi}} \int_{-\infty}^{\infty} \frac{e^{i t\omega}}{i\omega} dZ^*(\omega), \text{ the process becomes:}\\
&= Y(\lambda)
\end{align*}
Since $Y(t)$ is clearly a stationary process, $Y(\lambda)$ is also stationary by Lemma \ref{lemm:stationary}.

Furthermore, if we use the specific allowable measure defined by the first-order differencing operator, $\lambda_{\Delta_{\iota,t}} \in \Lambda_1$:
\begin{align*}
X_{\Delta_{\iota,t}}(t) : &= X(\lambda_{\Delta_{\iota,t}}) = Y(\lambda_{\Delta_{\iota,t}})\\
&= \int_{-\infty}^{\infty} \frac{1}{\sqrt{2\pi}} \frac{e^{i t \omega}}{i\omega} dZ^*(\omega) - \int_{-\infty}^{\infty} \frac{1}{\sqrt{2\pi}} \frac{e^{i (t-\iota) \omega}}{i\omega} dZ^*(\omega)\\
&= \int_{-\infty}^{\infty} \frac{1}{\sqrt{2\pi}}\frac{e^{i t \omega} \biggl\{ 1 - e^{-i \iota \omega} \biggl\}}{i \omega} dZ^*(\omega)
\end{align*}
This is precisely the spectral representation of a stationary process, where $\Delta_\iota X(t)$ is stationary (the I(1) definition). The corresponding structure function is:
\begin{align*}
Cov \bigl( X_{\Delta_{\iota,t}}(t), X_{\Delta_{\iota,t}}(t+h) \bigl) &= D(h) = \frac{1}{2\pi} \int_{-\infty}^{\infty} e^{i h \omega} \frac{(1-e^{-i \iota \omega}) (1-e^{i \iota \omega})}{\omega^2} dF^*(\omega)
\end{align*}
where $dF^*(\omega)= \langle|dZ^*(\omega)|^2\rangle$. Hence, the structure function $D(h)$ is stationary with respect to $h$, demonstrating the direct link between $IRF(1)$ and $I(1)$ via the use of the allowable measure.

\bigskip

\subsection{Universal Kriging}

Universal kriging is a commonly used application of an IRF(d) in spatial statistics, where the goal is to find the best linear unbiased estimator (BLUE), which minimizes the mean-squared prediction error, to predict values in unsampled locations. Let $X(t)$ be an intrinsic random process with order $d$ and $t \in \mathbb{R}$. Then, given the allowable measure $\lambda_{t_0} \in \Lambda_d$ such that $\lambda_{t_0} := \sum_{i=1}^n \eta_i \delta_{t_i} - \delta_{t_0}$, we can express the mean-squared prediction error for any unsampled index $t_0 \in \mathbb{R}$ as:

\begin{align*}
    & E\bigl( X(\lambda_{t_0}) \bigl)^2 = E\bigl( \sum_{i=1}^n \eta_i X(t_i) - X(t_0) \bigl)^2 \quad \text{where $t_1, t_2, \cdots, t_n \in \mathbb{R}$ are the observed indices from the dataset.} \\
    &\text{Then, by Theorem \ref{thm:irf_rpsi} and Remark \ref{rmk:spectral_decomp_conti},}\\
    &= E\biggl[ \sum_{i=1}^n \eta_i \int_{-\infty}^{\infty} \bigl( e^{i t_i \omega} -1 - i t_i \omega - \cdots - \frac{(i t_i \omega)^{d-1}}{(d-1)!} \bigl) dZ(\omega) - \int_{-\infty}^{\infty} \bigl( e^{i t_0 \omega} -1 - i t_0 \omega - \cdots - \frac{(i t_0 \omega)^{d-1}}{(d-1)!} \bigl) dZ(\omega) \biggl]^2\\
    &= E\Biggl[ \int_{-\infty}^{\infty} \Biggl\{ \sum_{i=1}^n \eta_i \bigl( e^{i t_i \omega} -1 - i t_i \omega - \cdots - \frac{(i t_i \omega)^{d-1}}{(d-1)!} \bigl) - \bigl( e^{i t_0 \omega} -1 - i t_0 \omega - \cdots - \frac{(i t_0 \omega)^{d-1}}{(d-1)!} \bigl) \Biggl\} dZ(\omega) \Biggl]^2\\
    &\text{Since } \lambda_{t_0} = \sum_{i=1}^n \eta_i \delta_{t_i} - \delta_{t_0} \in \Lambda_d,\\
    &= E\bigl[ \int_{-\infty}^{\infty} \bigl( \sum_{i=1}^n \eta_i  e^{i t_i \omega} - e^{i t_0\omega} \bigl) dZ(\omega) \bigl]^2 = E\bigl[ \sum_{i=1}^n \eta_i \int_{-\infty}^{\infty} e^{i t_i \omega} dZ(\omega) - \int_{-\infty}^{\infty} e^{i t_0 \omega} dZ(\omega) \bigl]^2\\
    &\text{Defining } Y^{**}(t) := \int_{-\infty}^{\infty} e^{i t \omega} dZ(\omega).\\
    &= E\bigl[ \sum_{i=1}^n \eta_i Y^{**}(t_i) - Y^{**}(t_0) \bigl]^2 = E\bigl[ Y^{**}(\lambda_{t_0}) \bigl]^2\\
\end{align*}
Therefore, considering the measurement of error $\sigma^2$ and the constraints from the allowable measure $\lambda_{t_0}$, the object function to minimize is
$$M(\underset{\sim}{\eta}) = \sigma^2 \underset{\sim}{\eta}^T \underset{\sim}{\eta} + \underset{\sim}{\eta}^T \Psi \underset{\sim}{\eta} - 2 \underset{\sim}{\eta}^T \underset{\sim}{\phi} + \phi(0) + 2(\underset{\sim}{\eta}^TQ - \underset{\sim}{q}^T) \underset{\sim}{\rho}$$
where
$\underset{\sim}{\eta}= \{ \eta_1, \eta_2, \cdots, \eta_n \}^T$, \quad $\Psi = \{ \phi(t_i - t_j) \}_{n \times n}$, \quad $\underset{\sim}{\phi} = \{ \phi(t_1-t_0), \phi(t_2-t_0), \cdots, \phi(t_n - t_0) \}^T$, \quad $\underset{\sim}{q} = \{ 1, t, t^2, \cdots, t^{d-1} \}^T$, \quad and \quad $Q = \underset{\sim}{q} \underset{\sim}{q}^T$.\\
The function $\phi(\cdot)$ is an intrinsic covariance function of $X(\cdot)$; that is, it is the stationary covariance function of $Y^{**}(\cdot)$. The $d \times 1$ vector $\underset{\sim}{\rho}$ is the Lagrange multiplier associated with the constraints on the allowable measure $\lambda_{t_0}$. As a result, we can obtain the coefficient for the kriging predictor such that

$$ \underset{\sim}{\eta} = (\Psi + \sigma^2 I)^{-1} \Bigl[ \underset{\sim}{\phi} + Q\bigl\{ Q^T ( \Psi + \sigma^2 I)^{-1} Q \bigl\}^{-1} \bigl\{ \underset{\sim}{q} - Q^T(\Psi + \sigma^2 I)^{-1} \underset{\sim}{\phi} \bigl\} \Bigl].$$

This demonstrates that the kriging estimate for intrinsic random functions can also be derived using the spectral representations of random processes with stationary increments on the real line. In other words, the established equivalence between the two non-stationary classes, ensures that the stationary covariance structure of the differenced process (the I(d)'s structure function) is the mathematical analog of the intrinsic covariance function ($\text{IRF(d)}$ generalized covariance) when acting upon the same order of process increments. This means modelers can confidently use spectral analysis and estimation techniques derived from I(d) literature to build and justify the covariance models used for Universal Kriging, regardless of whether the original process was conceived as an I(d) or IRF(d). The end result is a clear, integrated mathematical foundation that deepens our understanding and enables more robust utilization of these processes in non-stationary applications.

\bigskip

\section{Conclusion}\label{sec:conclusion}

The central result of this paper is the rigorous demonstration of equivalence between the Intrinsic Random Function of order $d$ ($\text{IRF(d)}$) and the Random Process with Stationary Increments of order $d$ ($\text{I(d)}$) on the real line. This unification confirms that the fundamental mathematical structure underlying these two prominent non-stationary modeling frameworks is identical under conditions of finite spectral energy. Specifically, the equivalence establishes that the dependency information captured by the $\text{I(d)}$ structure function is the sole stationary component necessary to define the $\text{IRF(d)}$'s intrinsic covariance. This congruence resolves historical ambiguities and strengthens the practical confidence in utilizing either methodology.

The established equivalence provides a clear, integrated mathematical foundation, significantly deepening our understanding and enabling more robust utilization of these essential non-stationary processes. Key practical benefits include the validation of geostatistical models (since generalized covariance is supported by spectral theory), improved spectral estimation (allowing advanced time series tools to be used in spatial statistics), and enhanced methodological cross-pollination, ultimately streamlining the analysis and modeling of complex non-stationary phenomena in both time series and spatial statistics.

\bigskip

\bibliographystyle{unsrt}  
\bibliography{references}

\end{document}